\documentclass[reqno]{amsart}
\usepackage[english]{babel}
\usepackage{amssymb}
\usepackage[T1]{fontenc}

\newtheorem{theorem}{Theorem}[section]

\newtheorem{lemma}[theorem]{Lemma}
\newtheorem{corollary}[theorem]{Corollary}

\theoremstyle{definition}
\newtheorem{definition}[theorem]{Definition}
\newtheorem{example}[theorem]{Example}

\theoremstyle{remark}
\newtheorem*{remark}{Remark}

\numberwithin{equation}{section}

\newcommand{\C}{\mbox{$\mathbb{C}$}}

\newcommand{\E}{\mbox{$\mathcal{E}$}}

\newcommand{\Ep}{\mbox{$\mathcal{E}_{p}$}}

\newcommand{\K}{\mbox{$\mathcal{K}$}}

\newcommand{\vogeq}{\mbox{$\,\succcurlyeq\,$}}

\newcommand{\vertiii}[1]{{\left\vert\kern-0.25ex\left\vert\kern-0.25ex\left\vert #1\right\vert\kern-0.25ex\right\vert\kern-0.25ex\right\vert}}

\begin{document}

\author{Per \AA hag}
\address{Department of Mathematics and Mathematical Statistics\\ Ume\aa \ University\\ SE-901 87 Ume\aa \\ Sweden}
\email{Per.Ahag@math.umu.se}
\author{Rafa\l\ Czy{\.z}}
\address{Institute of Mathematics\\ Jagiellonian University\\ \L ojasiewicza 6\\ 30-348 Krak\'ow\\ Poland}
\email{Rafal.Czyz@im.uj.edu.pl}
\keywords{complex Hessian operator, $m$-subharmonic function, integrability, Poincar\'e type inequality, Sobolev type inequality}
\subjclass[2010]{Primary 35J60, 46E35, 26D10; Secondary 32U05, 31C45.}
\title[Poincar\'e-- and Sobolev-- type inequalities]{Poincar\'e-- and Sobolev-- type inequalities for complex $m$-Hessian equations}

\begin{abstract} By using quasi-Banach techniques as key ingredient we prove Poincar\'e-- and Sobolev-- type inequalities for $m$-subharmonic functions with finite $(p,m)$-energy. A consequence of the Sobolev type inequality is a partial confirmation of B\l ocki's integrability conjecture for
$m$-subharmonic functions.
\end{abstract}

\maketitle

\begin{center}\bf
\today
\end{center}

\section{Background}

In 1985,  Caffarelli, Nirenberg, Spruck introduced the so called real $k$-Hessian operator, $S_k$, in  bounded domains in $\mathbb R^n$, $n\geq 2$, $1\leq k\leq n$ (\cite{CNS}). The real $k$-Hessian operator is a nonlinear partial differential operator acting on what is known as \emph{$k$-admissible} functions (also known as \emph{$k$-convex} functions). A $\mathcal C^2$-function $u$ is $k$-admissible if the following elementary symmetric functions are non-negative
\[
\sigma_l(\lambda (D^2u))=\sum_{1\leq j_j<\dots<j_k\leq n}\lambda_{j_1}\cdots\lambda_{j_k}, \ \text { for } \ l=1,\dots,k,
\]
where $\lambda(D^2u)=(\lambda_1,\dots,\lambda_n)$ are eigenvalues of the real Hessian matrix $D^2u=[\frac {\partial ^2u}{\partial x_j\partial x_i}]$. The real $k$-Hessian operator is then defined by
\[
S_k(u)=\sigma_k(\lambda(D^2u)).
\]
By these definitions we get that the $1$-Hessian operator is the classical Laplace operator defined on $1$-admissible functions that
are just the subharmonic functions. Furthermore, the $n$-Hessian operator is the real Monge-Amp\`ere operator defined on $n$-admissible functions that are the same as the convex functions. Therefore, for $k=2,\dots, n-1$, the $k$-Hessian operator  can be regarded as a sequence of nonlinear partial differential operators linking the classical Laplace operator to the real Monge-Amp\`ere operator. The natural progression is then to extend the set of $k$-admissible functions together with the real $k$-Hessian operator. This was done in the famous trilogy written by Trudinger and Wang~\cite{TrudingerWangI,TrudingerWangII,TrudingerWangIII} (especially~\cite{TrudingerWangII}).

For $k=1,..,n$, the \emph{$k$-Hessian integral} is formally defined as
\[
I_0(u)=\int_{\Omega} (-u) \quad \text{ and } \quad I_k(u)=\int_{\Omega} (-u)S_k(u).
\]
When $k=1$ we see that $I_1(u)=\int_{\Omega} |D u|^2$ is the Dirichlet energy integral from potential theory that goes back to the work of Gau\ss, Dirichlet, Riemann, among many others,
while
for $k=n$
\[
I_n(u)=\int_{\Omega} (-u) \det\big(D^2u\big)
\]
is the fundamental integral in the variational theory for the real Monge-Amp\`ere equations (see e.g.~\cite{Aubin,Bakelman61,Bakelman65,Bakelman,ChouTso90}). The $k$-Hessian integral was introduced by Chou~\cite{ChouTso89}. For further information about the $k$-Hessian integral see e.g.~\cite{wangSurvey}.

 Now let $0\leq l< k\leq n$, and let $\Omega\subset \mathbb R^n$, $n\geq 2$, be a smoothly bounded $(k-1)$-convex domain, and let $u$ be an $k$-admissible function that vanishes on $\partial \Omega$. Then there exists a constant $C(l,k,n,\Omega)$ depending only on $n,l,k$ and $\Omega$ such that
\begin{equation}\label{poinc_real}
I_l(u)^{\frac {1}{l+1}}\leq C(l,k,n,\Omega)\, I_k(u)^{\frac {1}{k+1}}.
\end{equation}
For $l=0$ and $k=1$, we have that inequality~(\ref{poinc_real}) is
\[
\int_{\Omega} (-u) \leq C(0,1,n,\Omega)  \left(\int_{\Omega}|D u|^2\right)^{1/2},
\]
and this can be interpreted as a type of the classical Poincar\'e inequality and therefore motivates calling~(\ref{poinc_real}) a Poincar\'e type inequality for $k$-Hessian operators. Inequality~(\ref{poinc_real}) was first proved by Trudinger and Wang~\cite{TrudingerWangPoincare} (for an alternative proof see~\cite{Hou}).

Under the same requirements on $\Omega$ and $u$ the Sobolev type inequality that is of our interest states then that there exists a constant $C(k,n,\Omega)$ depending only on $k,n$ and $\Omega$ such that:

\medskip

\begin{enumerate}\itemsep2mm
\item if $1\leq k<\frac n2$, then
\[
\|u\|_{L^{q}}\leq C(k,n,\Omega) \, I_k(u)^{\frac {1}{k+1}}, \quad \text{for} \quad 1\leq q\leq \frac{n(k+1)}{n-2k};
\]
\item if $k=\frac {n}{2}$, then
\[
\|u\|_{L^{q}}\leq C(k,n,\Omega) \, I_k(u)^{\frac {1}{k+1}}, \quad \text{for} \quad q<\infty;
\]
\item if $\frac n2<k\leq n$, then
\[
\|u\|_{L^{\infty}}\leq C(k,n,\Omega) \, I_k(u)^{\frac {1}{k+1}}.
\]
\end{enumerate}

\bigskip

\noindent If $k=1$, then we have
\[
\|u\|_{L^{q}}\leq C(1,n,\Omega)  \left(\int_{\Omega}|D u|^2\right)^{1/2}, \quad \text{for} \quad 1\leq q\leq \frac{2n}{n-2},
\]
and for $k=n$,
\[
\|u\|_{L^{\infty}}\leq C(n,n,\Omega) \left(\int_{\Omega} (-u) \det\big(D^2u\big)\right)^{\frac {1}{n+1}}.
\]
The Sobolev type inequalities $(1)$--$(3)$ for $n$-admissible functions was first proved by Chou~\cite{ChouTso89}, while for the general case they were proved by Wang~\cite{wangSobolev} (see also~\cite{Trudinger97}).

\medskip

Now to the complex setting. Let $n\geq 2$ and $1\leq m\leq n$. Mimicking the real case above we say that a $\mathcal C^2$-function $u$ defined in a bounded domain in $\mathbb C^n$ is \emph{$m$-subharmonic} or \emph{$m$-admissible} if the elementary symmetric functions are positive $\sigma_l(\lambda(u))\geq 0$  for $l=1,\dots,m$, where this time $\lambda(u)=(\lambda_1,\ldots,\lambda_n)$ are eigenvalues of the complex Hessian matrix $D_{\mathbb{C}}^2u=[\frac {\partial ^2u}{\partial z_j\partial \bar z_k}]$. The complex $m$-Hessian operator on a $\mathcal C^2$-function $u$ is then defined by
\[
\operatorname H_m(u)=\sigma_m(\lambda(D_{\mathbb{C}}^2u)).
\]
In the complex case we get that the complex $1$-Hessian operator is the classical Laplace operator defined on $1$-subharmonic functions that
are just the subharmonic functions, while the complex $n$-Hessian operator is the complex Monge-Amp\`{e}re operator defined on $n$-subharmonic functions that is the plurisubharmonic functions. An early encounter of the complex $m$-Hessian operator is the work of Vinacua~\cite{vinacua1} from 1986. That work was later published in article form in~\cite{vinacua2}. The extension of $m$-subharmonic functions and the complex $m$-Hessian operator to non-smooth admissible functions was done by B\l ocki in 2005 (\cite{Blocki05}). There he
also introduced pluripotential methods to the theory of complex Hessian operators. Standard notations and terminology in the real and complex case differ in part, and so instead of $ I_k $ above, we shall use the following notation in the complex case: For $p>0$, $p\in\mathbb{R}$, and $m=1,..,n$, let
\[
e_{0,m}(u)=\int_{\Omega} \operatorname H_m(u) \quad \text{ and } \quad e_{p,m}(u)=\int_{\Omega} (-u)^p \operatorname H_m(u),
\]
and we call $e_{p,m}(u)$ for the \emph{$(p,m)$-energy of $u$}. Thus, for $k=1$ we have that $e_{1,1}(u)=I_1(u)$, but notice the difference in the definition of $e_{0,1}(u)$ compared to $I_0(u)$. For the early work on the theory of variation for the complex $n$-Hessian operator
see e.g.~\cite{BedfordTaylor78,BedfordTaylor79,ChernLevineNirenberg,Gaveau1,Gaveau2,Kalina}.

To be able to prove the Poincar\'e- and Sobolev- type inequalities for $m$-subharmonic functions we need classes of $m$-subharmonic functions that, in a general sense, vanishes on the boundary and additionally they should have finite $(p,m)$-energy. Denote these classes with $\mathcal E_{p,m}(\Omega)$  (see Section~\ref{sec_prelimin} for details).

Our Poincar\'e type inequality for the complex $m$-Hessian operator is:

\medskip

\noindent \textbf{Theorem~\ref{Piep}.} \emph{Let $n\geq 2$, $1\leq l< k\leq n$, and $p\geq 0$. Assume that $\Omega$ is a bounded $B_k$-regular domain in $\mathbb C^n$. Then there exits a constant $C(p,l,k,n,\Omega)>0$, depending only on $p$, $l$, $k$, $n$, and $\Omega$, such that for any $u\in \mathcal E_{p,k}(\Omega)$ we have}
\begin{equation}\label{Piep_1}
e_{p,l}(u)^{\frac{1}{p+l}}\leq C(p,l,k,n,\Omega) e_{p,k}(u)^{\frac{1}{p+k}}.
\end{equation}

\medskip

\noindent If $p=0$, $l=1$, and $k=n$, then~(\ref{Piep_1}) becomes
\[
\begin{aligned}
\int_{\Omega} \Delta u &\leq C(0,1,n,n,\Omega) \left(\int_{\Omega} H_n(u)\right)^{\frac{1}{n}} =C(0,1,n,n\,\Omega) \left(\int_{\Omega} \det D_{\mathbb{C}}^2u\right)^{\frac{1}{n}} \\
&= C(0,1,n,n,\Omega) \left(\int_{\Omega}(dd^c u)^n\right)^{\frac{1}{n}},
\end{aligned}
\]
where $(dd^c u)^n$ is the standard notation for the complex Monge-Amp\`{e}re operator in pluripotential theory. Furthermore, if $p=1$, $l=1$, and $k=n$, then we have that
\[
\int_{\Omega} |D u|^2 \leq   C(1,1,n,n,\Omega)^2 \left(\int_{\Omega}(-u)(dd^c u)^n\right)^{\frac{2}{n}}\, .
\]
When $\Omega$ is assumed to have the stronger convexity property known as strongly $k$-pseudoconvexity, and $p=1$, then inequality~(\ref{Piep_1}) was proved by Hou~\cite{Hou}. In Theorem~\ref{Piep_optimal} we find
the optimal constant in~(\ref{Piep_1}) for the cases $p=0$, and $p=1$, in the case of the unit ball $\Omega=\mathbb B$.

\medskip

Our Sobolev type inequality for complex $m$-Hessian equations is:

\medskip

\noindent \textbf{Theorem~\ref{sob}.} \emph{Let $n\geq 2$, $1\leq m\leq n$, and $p\geq0$. Assume that $\Omega$ be a bounded $m$-hyperconvex domain in $\mathbb C^n$. There exists a constant $C(p,q,m,n,\Omega)>0$, depending only on $p$, $q$, $m$, $n$, and $\Omega$ such that for any function $u\in \mathcal E_{p,m}(\Omega)$, and for $0<q<\frac {(m+p)n}{n-m}$, we have}
\begin{equation}\label{sob_1}
\|u\|_{L^q}\leq C(p,q,m,n,\Omega) e_{p,m}(u)^{\frac {1}{m+p}}\, .
\end{equation}

\medskip

\noindent For $p=0$, we have for $m=1$ and $m=n$, respectively,
\[
\|u\|_{L^q}\leq C(0,q,1,n,\Omega) \int_{\Omega} \Delta u \qquad \text{for} \qquad 0<q<\frac {n}{n-1},
\]
and
\[
\|u\|_{L^q}\leq C(0,q,n,n,\Omega) \left(\int_{\Omega} (dd^c u)^n\right)^{\frac{1}{n}} \qquad \text{for } q>0.
\]
Furthermore, for $p=1$, we have for $m=1$ and $m=n$, respectively,
\[
\|u\|_{L^q}\leq C(0,q,1,n,\Omega) \int_{\Omega} |D u|^2 \qquad \text{for} \qquad 0<q<\frac {2n}{n-1},
\]
and
\[
\|u\|_{L^q}\leq C(0,q,n,n,\Omega) \left(\int_{\Omega} (-u)(dd^c u)^n\right)^{\frac{1}{n}} \qquad \text{for } q>0.
\]

For the complex $n$-Hessian operator with $p=1$, inequality~(\ref{sob_1}) was proved by Berman and Berndtsson in~\cite{BB1} (see also~\cite{GKY}). Later their result was generalized by the authors to the case when $p$ is any positive number, and when $\Omega$ is a $n$-hyperconvex domain in $\mathbb C^n$, or a compact K\"ahler manifold (\cite{ACmoser}). The case when $\Omega$ is assumed to have the stronger convexity assumption of strongly $k$-pseudoconvexity, and $p=1$, then inequality~(\ref{sob_1}) was proved by Zhou~\cite{Zhou}.

After proving Theorem~\ref{sob} we give examples that shows that the following inequalities are not in general possible:
\[
\begin{aligned}
e_{p,m}(u)^{\frac {1}{m+p}}& \leq C\|u\|_{L^q} \ (\text{Example}~\ref{ex1})\\
\|u\|_{L^{\infty}} &\leq C e_{p,m}(u)^{\frac {1}{m+p}} \ (\text{Example}~\ref{ex2}) \\
e_{p,m}(u)^{\frac {1}{n+p}}& \leq C \|u\|_{L^{\infty}} \ (\text{Example}~\ref{ex3})\, .
\end{aligned}
\]

\medskip

It is well known that all $n$-subharmonic functions are locally in $L^p$ for any $p>0$. In general, this fact is no longer valid for $m$-subharmonic functions. B\l ocki proved that if $u$ is $m$-subharmonic function, then $u\in L^p_{loc}$ for $p<\frac {n}{n-m}$. Motivated by the real case he then conjectured that any $m$-subharmonic function is in $L^p_{loc}$ for $p<\frac {nm}{n-m}$ (\cite{Blocki05}). Later, Dinew and Ko\l odziej partially confirmed this conjecture under the extra assumption that the $m$-subharmonic functions (\cite{DinewKolodziej}). For the relation of this conjecture with the so called integrability exponent, and the Lelong number, of $m$-subharmonic functions see~\cite{BG}. As an immediate consequence of our Theorem~\ref{sob} is that we get that B\l ocki's conjecture is true for functions in the Cegrell class $\mathcal E_m(\Omega)$ (Corollary~\ref{cor}). The inequalities under investigation are very helpful in solving the Dirichlet problem for the complex Hessian type equation, and the solution of those equations can be used for the construction of certain metrics on compact K\"ahler and Hermitian manifolds (see e.g.~\cite{BB1,GKY}). Furthermore, the optimal constant in these inequalities are connected to the isoperimetric inequality and therefore classically to symmetrization of functions (see e.g.~\cite{T}).

\medskip

\noindent Both our proofs of Theorem~\ref{Piep}, and Theorem~\ref{sob}, uses the theory of quasi-Banach spaces (Theorem~\ref{qB}).

\section{Preliminaries}\label{sec_prelimin}

 Here we give some necessary background. We start with the definition of $m$-subharmonic functions and the $m$-Hessian operator.  Let $\Omega \subset \C^n$, $n\geq 2$, be a bounded domain, $1\leq m\leq n$, and define $\mathbb C_{(1,1)}$ to be the set of $(1,1)$-forms with constant coefficients. With this set
\[
\Gamma_m=\left\{\alpha\in \mathbb C_{(1,1)}: \alpha\wedge \beta^{n-1}\geq 0, \dots , \alpha^m\wedge \beta ^{n-m}\geq 0   \right\}\, ,
\]
where $\beta=dd^c|z|^2$ is the canonical K\"{a}hler form in $\C^n$.

\begin{definition}\label{m-sh} Let $n\geq 2$, and $1\leq m\leq n$. Assume that $\Omega \subset \C^n$ is a bounded domain, and let $u$ be a subharmonic function defined on $\Omega$. Then we say that $u$ is \emph{$m$-subharmonic} if the following inequality holds
\[
dd^cu\wedge\alpha_1\wedge\dots\wedge\alpha_{m-1}\wedge\beta^{n-m}\geq 0\, ,
\]
in the sense of currents for all $\alpha_1,\ldots,\alpha_{m-1}\in \Gamma_m$. With $\mathcal{SH}_m(\Omega)$ we denote the set of all $m$-subharmonic functions defined on $\Omega$.
\end{definition}

Let $\sigma_{k}$ be $k$-elementary symmetric polynomial of $n$-variable, i.e.,
\[
\sigma_{k}(x_1,\ldots,x_n)=\sum_{1\leq j_1<\cdots<j_k\leq n}x_{j_1}\cdots x_{j_k}\, .
\]
It can be proved that a function $u\in\mathcal C^2(\Omega)$ is $m$-subharmonic if, and only if,
\[
\sigma_k(u(z))=\sigma_k(\lambda_1(z),\ldots,\lambda_n(z))\geq 0,
\]
for all $k=1,\dots,m$, and all $z\in \Omega$. Here, $\lambda_1(z),\ldots,\lambda_n(z)$ are the eigenvalues of the complex Hessian matrix $\left [\frac {\partial ^2u}{\partial z_j\partial \bar z_k}(z)\right]$. For $\mathcal C^2$ smooth $m$-subharmonic function $u$, the \emph{complex $m$-Hessian operator} is defined by
\[
\operatorname{H}_m(u)=(dd^cu)^m\wedge(dd^c|z|^2)^{n-m}=4^nm!(n-m)!\sigma_m(u(z))dV_{2n},
\]
where $dV_{2n}$ is the Lebesgue measure in $\mathbb C^n$.

To be able to have sufficiently many $m$-subharmonic functions that vanishes in some sense on the boundary we need some
suitable convexity condition on our underlying domain. In this paper we need $m$-hyperconvexity (Definition~\ref{prel_hcx}),
and $B_m$-regularity (Definition~\ref{prel_breg}).

\begin{definition}\label{prel_hcx} Let $n\geq 2$, and $1\leq m\leq n$. A bounded domain in $\Omega\subset\C^n$ is said to be  \emph{$m$-hyperconvex} if it admits a non-negative and $m$-subharmonic exhaustion function, i.e.,  there exits a $m$-subharmonic $\varphi:\Omega\to [0,\infty)$ such that the closure of the set $\{z\in\Omega : \varphi(z)<c\}$ is compact in $\Omega$, for every $c\in (-\infty, 0)$.
\end{definition}

\begin{definition}\label{prel_breg}
Let $n\geq 2$, and $1\leq m\leq n$. A bounded domain in $\Omega\subset\C^n$ is said to be  \emph{$B_m$-regular} if for every $f\in \mathcal C(\partial \Omega)$ there exists a $m$-subharmonic function defined on $\Omega$ such that $u=f$ on $\partial \Omega$.
\end{definition}

\begin{remark}
\begin{enumerate}\itemsep2mm

\item $n$-hyperconvex domains are hyperconvex domains from pluripotential theory, while $1$-hyperconvex domains are regular domains in potential theory.

\item $B_n$-regular domains are $B$-regular domains from pluripotential theory, while $B_1$-regular domains are regular domains in  potential theory.

\item Every $B_m$-regular domain is $m$-hyperconvex. On the other hand, the bidisc $\mathbb{D}\times\mathbb{D}$ in $\mathbb{C}^2$ is $2$-hyperconvex, but not $B_2$-regular while it is both $1$-hyperconvex and $B_1$-regular.

\end{enumerate}

For proofs, and further information about these convexity notions see~\cite{ACH}.
\end{remark}

Next, we shall recall the function classes that are of our interest. As said in the introduction we shall use the following notations:
\[
e_{0,m}(u)=\int_{\Omega} \operatorname H_m(u) \quad \text{ and } \quad e_{p,m}(u)=\int_{\Omega} (-u)^p \operatorname H_m(u),
\]
We say that a $m$-subharmonic function $\varphi$ defined on a $m$-hyperconvex domain  $\Omega$ belongs to $\mathcal E^0_{m}(\Omega)$ if $\varphi$ is bounded,
\[
\lim_{z\rightarrow\xi} \varphi (z)=0 \quad \text{ for every } \xi\in\partial\Omega\, ,
\]
and
\[
\int_{\Omega} \operatorname{H}_m(\varphi)<\infty\, .
\]
\begin{definition} Let $n\geq 2$, $1\leq m\leq n$, and $p\geq0$. Assume that $\Omega$ be a bounded $m$-hyperconvex domain in $\mathbb C^n$.
We say that $u\in \E_{p,m}(\Omega)$, if $u$ is a $m$-subharmonic function defined on $\Omega$ such that there exists a decreasing sequence, $\{\varphi_{j}\}$, $\varphi_{j}\in\mathcal E^0_{m}(\Omega)$, that converges pointwise to $u$ on $\Omega$,
as $j$ tends to $\infty$, and $\sup_{j} e_{p,m}(\varphi_j)< \infty$.
\end{definition}

In~\cite{L1,L2}, it was proved that for $u\in \E_{p,m}(\Omega)$ the complex Hessian operator, $\operatorname H_m(u)$, is well-defined, and
\[
\operatorname{H}_m(u)=(dd^cu)^m\wedge(dd^c|z|^2)^{n-m}\, ,
\]
where $d=\partial +\bar{\partial}$, and $d^c=\sqrt{-1}\big(\bar{\partial}-\partial\big)$.

Theorem~\ref{thm_holder} is essential when working with $\E_{p,m}(\Omega)$, $p>0$.

\begin{theorem}\label{thm_holder} Let $n\geq 2$, $1\leq m\leq n$, and $p>0$. Assume that $\Omega$ be a bounded $m$-hyperconvex domain in $\mathbb C^n$. For $u_0,u_1,\ldots , u_m\in\E_{p,m}(\Omega)$ we have
\begin{multline*}
\int_\Omega (-u_0)^p dd^c u_1\wedge\cdots\wedge dd^c u_m\wedge (dd^c|z|^2)^{n-m}\\ \leq
C\; e_p(u_0)^{p/(p+m)}e_p(u_1)^{1/(p+m)}\cdots
e_p(u_m)^{1/(p+m)}\, ,
\end{multline*}
where $C\geq 1$ depends only on $p,m,n$ and $\Omega$.
\end{theorem}
\begin{proof}
See e.g. Lu~\cite{L1,L2}, and Nguyen~\cite{Thien16}.
\end{proof}

\begin{remark} If $p\neq 1$, then $C>1$ (see~\cite{ACmod,ACbeta,CzyzThien}).
\end{remark}

\section{quasi-Banach spaces}\label{sec_qBs}

In this section we introduce the necessary background of the theory of quasi-Banach spaces to be able to prove Theorem~\ref{qB} which subsequently will be used in both the proof of Theorem~\ref{Piep} and Theorem~\ref{sob}. Let $X$ be a real vector space. We say that $\K$ is a \emph{cone} in the vector space $X$ if it is a non-empty subset of $X$ that satisfies:
\begin{enumerate}\itemsep2mm

\item $\K+\K\subseteq\K$\, ,

\item $\alpha\K\subseteq\K$ for all $\alpha\geq 0$\, ,  and

\item $\K\cap (-\K)=\{0\}$.

\end{enumerate}
It should be noted that in some texts the name \emph{proper convex cone} is used instead. Furthermore, $\delta \K=\K-\K$ is vector subspace of $X$. Let us recall the definition of a quasi-norm and a quasi-Banach space.

\begin{definition}\label{quasi_def_quasi}
A \emph{quasi-norm} $\|\cdot\|_0$ on a cone $\K$ is a mapping
$\|\cdot\|_0:\K\to [0,\infty)$ with the following properties:
\begin{enumerate}

\item $\|x\|_0=0$ if, and only if, $x=0$;

\item $\|tx\|_0=t\|x\|_0$  for all $x\in \K$ and $t\geq 0$;

\item there exists a constant $C\geq 1$ such that for all
$x,y\in \K$ we have that
\begin{equation}\label{quasi_def_quasi_1}
\|x+y\|_0\leq C(\|x\|_0+\|y\|_0)\, .
\end{equation}
\end{enumerate}
The constant $C$ in~(\ref{quasi_def_quasi_1}) is often refereed to the modulus
of concavity of the quasi-norm $\|\cdot\|$. Now one can extend $\|\cdot\|_0$ to the vector space $\delta \K$ by
\[
\|x\|=\inf\left \{\|x_1+x_2\|_0: x=x_1-x_2, x_1,x_2\in \K\right\}\, .
\]
\end{definition}

The classical Aoki-Rolowicz theorem for quasi-Banach spaces (\cite{aoki,rolowicz}) states that
every quasi-normed space $X$ is $q$-normable for some $0<q\leq 1$. In other words, $X$ can
be endowed with an equivalent quasi-norm $\vertiii{\cdot}$ that is $q$-subadditive, and
therefore we can define the following metric $d(x,y)=\vertiii{x-y}^q$ on $X$. The vector space $X$ is
called a \emph{quasi-Banach space} if it is complete with respect to the metric $d$ induced
by the quasi-norm $\|\cdot\|$. Note that it follows from the definition of quasi-norm that for any $x_1,\dots,x_k\in \delta\K$ holds
\begin{equation}\label{qnineq}
\|x_1+\cdots +x_k\|\leq \sum_{j=1}^kC^j\|x_j\|.
\end{equation}
The cone $\K$ in a vector space $X$ generates a vector ordering $\vogeq$ defined on $\delta \K$ by letting $x\vogeq y$ whenever $x-y\in\K$.

\begin{theorem}\label{qB}
Let $X$ be a real vector space, $\K\subset X$ a cone, and let $\|\cdot\|_0$ be a quasi-norm on $\K$, such that $(\delta \K,\|\cdot\|)$ is a quasi-Banach space, and $\K$ is closed in $\delta\K$. Assume that $\Psi:X\to [0,\infty]$ is a function that satisfies:
\begin{itemize}\itemsep2mm

\item[$a)$] $\Psi$ is homogeneous, i.e. $\Psi(tx)=t\Psi(x)$, $t\geq 0$, $x\in \K$;

\item[$b)$] $\Psi$ is increasing, i.e. if $x\vogeq y$, then $\Psi(x)\geq \Psi(y)$.
\end{itemize}
The following conditions are then equivalent:
\begin{enumerate}

\item there exists a constant $B>0$ such that for all $x\in \K$ holds
\[
\Psi(x)\leq B \|x\|_0;
\]

\item $\Psi$ is finite on $\K$.
\end{enumerate}
\end{theorem}
\begin{proof} The implication (1)$\Rightarrow$(2) is clear. To prove the opposite implication (2)$\Rightarrow$(1) we shall argue by  contradiction. Assume that there does not exists any constant $B$ as above.  Therefore, by using homogeneity of $\Psi$ we can assume that there exists a sequence $x_j\in \K$ such that
\begin{equation}\label{1}
\|x_j\|_0=1 \qquad \text {and} \qquad \Psi(x_j)>j(2C)^j,
\end{equation}
where $C$ is  the modulus of concavity of the quasi-norm $\|\cdot\|_0$. Let us define
\[
y_k=\sum_{j=1}^k(2C)^{-j}x_j.
\]
We shall prove that $\{y_k\}$ is a Cauchy sequence. By (\ref{qnineq}) we have  that for $k>l$
\[
\|y_k-y_l\|_0=\|\sum_{j=l+1}^k(2C)^{-j}x_j\|_0\leq \sum_{j=l+1}^kC^j(2C)^{-j}\|x_j\|_0\leq \sum_{j=l+1}^k2^{-j}<2^{-l}.
\]
Therefore, there exists $y\in \delta\K$ such that $y_k\to y$, as $k\to \infty$. But since the cone $\K$ is closed we get that $y\in \K$.

On the other hand,  by the same argument as above we get that for any $m\in \mathbb N$ we have
\[
y=\sum_{j=1}^{\infty}(2C)^{-j}x_j\vogeq (2C)^{-m}x_m,
\]
and therefore by (\ref{1}) and monotonicity of $\Psi$
\[
\Psi(y)=\Psi\left(\sum_{j=1}^{\infty}(2C)^{-j}x_j\right)\geq \Psi\left((2C)^{-m}x_m\right)=(2C)^{-m}\Psi\left(x_m\right)>m.
\]
This is impossible by our assumption.
\end{proof}

\begin{remark}
Note that condition $b)$ in Theorem~\ref{qB} can be replaced by upper semicontinuity of $\Psi$.
\end{remark}

 We shall give some examples of Theorem~\ref{qB}. Example~\ref{e1}, and Example~\ref{e2}, shall be used in the proofs of Theorem~\ref{Piep}, and Theorem~\ref{sob}.

\begin{example}\label{e1} Assume that $\Omega$ is a bounded $m$-hyperconvex domain in $\mathbb C^n$. Let $X=L^1_{loc}(\Omega)$, $\K=\mathcal E_{p,k}(\Omega)$, and for $u\in \K$ let
\[
\|u\|_0=e_{p,k}(u)^{\frac {1}{p+k}}.
\]
Then for any $v\in \delta\mathcal E_{p,k}(\Omega)$ define
\[
\|u\|=\inf_{u_1-u_2=u \atop u_1,u_2\in \mathcal{E}_{p,k}}\left(\int_{\Omega} (-(u_1+u_2))^{p}\operatorname{H}_k(u_1+u_2)
\right)^{\frac {1}{k+p}}\, .
\]
It was proved in~\cite{ACmod,Thien16} that $(\delta\mathcal{E}_{p,k}, \|\cdot\|)$ is a quasi-Banach space for $p\neq 1$, and a Banach space for $p=1$. Furthermore, the cone $\mathcal{E}_{p,k}(\Omega)$ is closed in $\delta\mathcal{E}_{p,k}(\Omega)$.

Let $\mu$ be a  positive Radon measure $\mu$, and $p>0$. Then we define
\[
\Psi_1(u)=\left(\int_{\Omega}(-u)^p\,d\mu\right)^{\frac 1p}.
\]
The functional $\Psi_1$ satisfies conditions $a)$ and $b)$ in Theorem~\ref{qB}. This example
will be used in our proof of the Sobolev type inequality (Theorem~\ref{sob}). In this special case Theorem~\ref{qB} was proved by Cegrell, see~\cite{CegrellPCE}, and Lu~\cite{L1,L2}.

Inspired by $\Psi_1$, we define for $1\leq l\leq n$ the following:
\[
\Psi_2(u)=\left(\int_{\Omega}(-u)^p\operatorname{H}_l(u)\right)^{\frac {1}{p+l}}.
\]
This functional, $\Psi_2$, shall be used in the proof of the Poincar\'e type inequality (Theorem~\ref{Piep}).
\hfill{$\Box$}
\end{example}

\begin{example}\label{e2}
Let $\Omega$ be a bounded $m$-hyperconvex domain in $\mathbb C^n$, $n\geq 2$.  Also, let $X=L^1_{loc}(\Omega)$, $\K=\mathcal E_{0,k}(\Omega)$, and for
$u\in \K$ set
\[
\|u\|_0=e_{0,k}(u)^{\frac {1}{k}}.
\]
Then for any $v\in \delta\mathcal E_{0,k}(\Omega)$ define
\[
\|u\|=\inf_{u_1-u_2=u \atop u_1,u_2\in \mathcal{E}_{0,k}}\left(\int_{\Omega} \operatorname{H}_k(u_1+u_2)
\right)^{\frac {1}{k}}\, .
\]
It was proved in~\cite{CegrellWiklund,Thien16} that $(\delta\mathcal{E}_{0,k}(\Omega), \|\cdot\|)$ is a Banach space. Furthermore, the cone $\mathcal{E}_{0,k}(\Omega)$ is closed in $\delta\mathcal{E}_{0,k}(\Omega)$. In the proof of the Poincar\'e type inequality (Theorem~\ref{Piep}) we shall use the following functional ($1\leq l\leq n$):
\[
\Psi_3(u)=\left(\int_{\Omega}\operatorname{H}_l(u)\right)^{\frac 1l}.
\] \hfill{$\Box$}
\end{example}

\begin{example}
Let $n\geq 2$, $p>0$, and $1\leq m\leq n$. Furthermore, assume that  $\Omega$ is a bounded $m$-hyperconvex domain in $\mathbb C^n$, and let $X=\delta \mathcal {M}_{p,m}$, where
\begin{align*}
\mathcal {M}_{p,m}=\big\{\mu \; :\;  & \mu \text { is a non-negative Radon measure on $\Omega$ such that }\\
 & \operatorname{H}_m(u)=\mu \text{ for some } u_{}\in \mathcal E_{p,m} (\Omega)\big \}\, .
\end{align*}
Let us here recall that the following conditions are equivalent:
\begin{itemize}\itemsep2mm
\item [$(1)$] there exists a unique function $u\in \mathcal E_{p,m} (\Omega)$ such that $\operatorname{H}_m(u)=\mu$;
\item [$(2)$] there exists a constant $C\geq 0$ such that
\[
\int_{\Omega}(-u)^p\,d\mu\leq C \left(e_{p,m}(u)\right)^{\frac {p}{p+m}} \;\text{ for all } u\in\mathcal E_{m}^0(\Omega)\, ;
\]
\item [$(3)$] $\mathcal E_{p,m}(\Omega) \subset L^p(\mu)$.

\end{itemize}
(\cite{ACP,CegrellPCE,L2,Thien16}). For $\mu\in\delta\mathcal M_{p,m}$ let now $u^{+}, u^{-}\in \mathcal E_{p,m}(\Omega)$ be the unique
$m$-subharmonic functions such that
\[
\operatorname{H}_m(u^{+})=\mu^{+}=\frac 12(|\mu|+\mu), \ \ \text { and } \ \ \operatorname{H}_m(u^{-})=\mu^{-}=\frac 12(|\mu|-\mu)\, .
\]
Now we can define
\begin{equation*}
|\mu|_{p,m}=\|u^{+}\|_{p,m}^m+\|u^{-}\|_{p,m}^m\, .
\end{equation*}
Then it was proved in~\cite{ACmod,Thien16} that $(\delta\mathcal M_{p,m}, |\cdot|_{p,m})$ is a quasi-Banach space, and for $p=1$ a Banach space.

In this space one can consider the following functional: For $p>0$, and a $m$-subharmonic function $u$ define
\[
\Psi_4(\mu)=\int_{\Omega}(-u)^p\,d\mu.
\]
The functional, $\Psi_4$, satisfies conditions $a)$ and $b)$ in Theorem~\ref{qB}. In this special case Theorem~\ref{qB} was proved in~\cite{AC2019ann} in order to characterize $\mathcal E_{p,k}(\Omega)$ \hfill{$\Box$}
\end{example}

\section{A Poincar\'e type inequality in $B_k$-regular domains}

The aim of this section is to prove the  Poincar\'e type inequality in $B_k$-regular domains for $k$-subharmonic functions. First we need the following lemma.

\begin{lemma}\label{lemmaep}
Let $n\geq 2$, $1\leq l< k\leq n$, and $p\geq 0$. Furthermore, assume that $\Omega$ is a bounded $B_k$-regular domain in $\mathbb C^n$.   Then $\mathcal E_{p,k}(\Omega)\subset \mathcal E_{p,l}(\Omega)$.
\end{lemma}
\begin{proof}
Let $u\in \mathcal E^0_{k}(\Omega)$. Since  $\Omega$ is $B_k$-regular we know that there exists a negative, smooth, $k$-subharmonic function $\varphi\in \mathcal E^0_{k}(\Omega)$ such that $(\varphi(z)-|z|^2)\in \mathcal {SH}_k(\Omega)$. Then define
\[
\mu:=(dd^cu)^l\wedge (dd^c|z|^2)^{n-l}.
\]
Then we have
\begin{equation}\label{kl}
\begin{aligned}
\mu&=(dd^cu)^l\wedge (dd^c|z|^2)^{n-l}\leq (dd^cu)^l\wedge\left(dd^c\left((\varphi-|z|^2)+|z|^2\right)\right)^{k-l}\wedge (dd^c|z|^2)^{n-k}\\
&=(dd^cu)^l\wedge\left(dd^c\varphi\right)^{k-l}\wedge (dd^c|z|^2)^{n-k}\leq \left(dd^c(u+\varphi)\right)^{k}\wedge (dd^c|z|^2)^{n-k}.
\end{aligned}
\end{equation}
Since $(u+\varphi)\in \mathcal E^0_{k}(\Omega)$ it follows that $\operatorname{H}_m(u+\varphi)$ is a finite measure, and therefore $\mu$ is also finite and $u\in \mathcal E^0_{l}(\Omega)$. Hence, $\mathcal E^0_{k}(\Omega)\subset \mathcal E^0_{l}(\Omega)$.

\medskip

\emph{Case $p>0$:} Assume that $u\in \mathcal E_{p,k}(\Omega)$. Then by definition there exists a decreasing sequence $u_j\in \mathcal E^0_{k}(\Omega)$
such that
\[
\lim_{j\to \infty}u_j=u\qquad \text{ and } \qquad \sup_{j}e_{p,k}(u_j)<\infty.
\]
Hence, $u_j\in \mathcal E^0_{l}(\Omega)$, and by Theorem~\ref{thm_holder} and (\ref{kl}) we get
\[
\begin{aligned}
&\int_{\Omega}(-u_j)^p(dd^cu_j)^l\wedge(dd^c|z|^2)^{n-l}\leq \int_{\Omega}(-u_j)^p(dd^c(u_j+\varphi))^{k}\wedge (dd^c|z|^2)^{n-k}\\
&\leq d(p,m,\Omega)e_{p,k}(u_j)^{\frac {p}{p+k}}e_{p,k}(u_j+\varphi)^{\frac {k}{p+k}}\\
&\leq d(p,m,\Omega)e_{p,k}(u_j)^{\frac {p}{p+k}}\left(d(p,m,\Omega)\left(e_{p,k}(u_j)^{\frac {1}{p+k}}+e_{p,k}(\varphi)^{\frac {1}{p+k}}\right)\right)^k.
\end{aligned}
\]
Hence, $\sup_{j}e_{p,l}(u_j)<\infty$. Thus, $u\in \mathcal E_{p,l}(\Omega)$.

\medskip

\emph{Case $p=0$:} Assume that $u\in \mathcal E_{0,k}(\Omega)$.  By definition there exists a decreasing sequence
$u_j\in \mathcal E^0_{k}(\Omega)$ such that
\[
\lim_{j\to \infty}u_j=u \qquad \text{ and } \qquad \sup_{j}e_{0,k}(u_j)<\infty.
\]
Hence, $u_j\in \mathcal E^0_{l}$ and therefore by~\cite{HungNhuyen} and (\ref{kl}) we get
\[
\begin{aligned}
\int_{\Omega}(dd^cu_j)^l\wedge(dd^c|z|^2)^{n-l}&\leq \int_{\Omega}(dd^c(u_j+\varphi))^{k}\wedge (dd^c|z|^2)^{n-k}\\
&\leq \left(e_{0,k}(u_j)^{\frac {1}{k}}+e_{0,k}(\varphi)^{\frac {1}{k}}\right)^k.\\
\end{aligned}
\]
This means that $\sup_{j}e_{0,l}(u_j)<\infty$, so $u\in \mathcal E_{0,l}(\Omega)$.
\end{proof}

\begin{remark} Let $\Omega=\mathbb D^2$ be the bidisc in $\mathbb C^2$. This domain is $2$-hyperconvex, but not $B_2$-regular. Let
\[
u(z_1,z_2)=\sum_{k=1}^{\infty}\max(\log |z_1|, k^{-4}\log|z_2|)
\]
be defined on $\Omega$. Then $u\in\mathcal E_{0,2}(\Omega)$, but it is not in $\mathcal E_{0,1}(\Omega)$ (see~\cite{CegrellWiklund}
for details). Next, define
\[
v(z_1,z_2)=\sum_{j=1}^{\infty}\max (j^{-6}\ln |z_1|,\ln |z_2|,-1)\, .
\]
By straight forward calculations we see that $v\in \mathcal E_{0,2}(\Omega)\cap\mathcal E_{1,2}(\Omega)$, but it is not in $\mathcal E_{0,1}(\Omega)\cup \mathcal E_{1,1}(\Omega)$.
\end{remark}

Now to the proof of the Poincar\'e type inequality.

\begin{theorem}\label{Piep}
Let $n\geq 2$, $1\leq l< k\leq n$, and $p\geq 0$. Assume that $\Omega$ is a bounded $B_k$-regular domain in $\mathbb C^n$. Then there exits a constant $C(p,l,k,n,\Omega)>0$, depending only on $p$, $l$, $k$, $n$, and $\Omega$, such that for any $u\in \mathcal E_{p,k}(\Omega)$ we have
\begin{equation*}
e_{p,l}(u)^{\frac{1}{p+l}}\leq C(p,l,k,n,\Omega) e_{p,k}(u)^{\frac{1}{p+k}}.
\end{equation*}
\end{theorem}
\begin{proof}
Using the functionals $\Psi_2$ and $\Psi_3$ (from Example~\ref{e1} and Example~\ref{e2}) the proof follows from Theorem~\ref{qB} and Lemma~\ref{lemmaep}.
\end{proof}

Next, we shall determine the optimal constant in Theorem~\ref{Piep} for the unit ball in $\mathbb C^n$ in the cases $p=0$ and $p=1$.

\begin{theorem}\label{Piep_optimal}
Let $n\geq 2$, $1\leq l< k\leq n$, and $\mathbb B$ be the unit ball in $\mathbb C^n$. The optimal constant $C(p,l,k,n,\mathbb B)$ in Theorem~\ref{Piep} is given by:

\medskip

\begin{itemize}\itemsep3mm

\item[$a)$] $\displaystyle{C(0,l,k,n,\mathbb B)=\left(4\pi\right)^{\frac nl-\frac nk}}$  $\,(p=0)$;

\item[$b)$] $\displaystyle{C(1,l,k,n,\mathbb B)=\left(\frac {(4\pi)^n}{n+1}\right)^{\frac 1l-\frac 1k}}$ $\,(p=1)$.

\end{itemize}
\end{theorem}

\begin{proof}

\emph{Case $p=0$:} We shall start proving that there exists a constant $C>0$ such that for any $u\in \mathcal E_{0,m}(\mathbb B)$ it
holds
\begin{equation}\label{epball2}
\left(\int_{\mathbb B}(dd^cu)^l\wedge(dd^c|z|^2)^{n-l}\right)^{\frac {1}{l}}\leq C\left(\int_{\mathbb B}(dd^cu)^k\wedge(dd^c|z|^2)^{n-k}\right)^{\frac {1}{k}}.
\end{equation}
Set $\beta=dd^c(|z|^2-1)$, and note that $|z|^2-1$ is an exhaustion function for $\mathbb B$. Then for any $u\in\mathcal E_{0,m}(\mathbb B)$. We get by~\cite{HungNhuyen}
\begin{multline*}
\int_{\mathbb B}(dd^cu)^{l}\wedge (dd^c|z|^2)^{n-l}=\int_{\mathbb B}(dd^cu)^l\wedge (dd^c(|z|^2-1))^{k-l}\wedge (dd^c|z|^2)^{n-k}\\
\leq \left(\int_{\mathbb B}(dd^cu)^{k}\wedge (dd^c|z|^2)^{n-k}\right)^{\frac {l}{k}} \left(\int_{\mathbb B}(dd^c|z|^2)^{n}\right)^{\frac {k-l}{k}},
\end{multline*}
and therefore
\[
\left(\int_{\mathbb B}(dd^cu)^l\wedge(dd^c|z|^2)^{n-l}\right)^{\frac {1}{l}}\leq C\left(\int_{\mathbb B}(dd^cu)^k\wedge(dd^c|z|^2)^{n-k}\right)^{\frac {1}{k}}.
\]
Thus,
\[
C(0,l,k,n,\mathbb B)=\left(\int_{\mathbb B}(dd^c|z|^2)^{n}\right)^{\frac 1l-\frac {1}{k}}=\left(4\pi\right)^{\frac nl-\frac nk}.
\]
This constant is optimal since we have equality in the Poincr\`e type inequality for the function $|z|^2-1$.
\medskip

\emph{Case $p=1$:} As in the case above we shall set $\beta=dd^c(|z|^2-1)$, and note that $|z|^2-1$ is an exhaustion function for the unit ball $\mathbb B$. Let $u\in\E_{p,k}(\mathbb B)$, $p>0$. Then by using H\"{o}lder's inequality, Theorem~\ref{thm_holder} and integration by parts we get
\[
\begin{aligned}
&e_{p,k-1}(u)=\int_{\mathbb B}(-u)^p(dd^cu)^{k-1}\wedge \beta^{n-k+1} \\
&=\int_{\mathbb B}(1-|z|^2)dd^c\left (-(-u)^p\right)\wedge(dd^cu)^{k-1}\wedge \beta^{n-k} \\
&\leq p\int_{\mathbb B}(1-|z|^2)(-u)^{p-1}(dd^cu)^k\wedge \beta^{n-k}\\
&\leq p\left(\int_{\mathbb B}(-u)^p(dd^cu)^{k}\wedge \beta^{n-k}\right)^{\frac {p-1}{p}}\times \left(\int_{\mathbb B}(1-|z|^2)^p(dd^cu)^{k}\wedge \beta^{n-k}\right)^{\frac {1}{p}}\\
&\leq p\, e_{p,k}(u)^{\frac {p-1}{p}}d(p,k,\mathbb B)^{\frac 1p}e_{p,k}(u)^{\frac {k}{(p+k)p}}e_{p,k}(|z|^2-1)^{\frac {p}{(p+k)p}}\\
&=p\, d(p,k,\mathbb B)^{\frac 1p}e_{p,k}(|z|^2-1)^{\frac {1}{p+k}}e_{p,k}(u)^{\frac {p+k-1}{p+k}}.
\end{aligned}
\]
Hence,
\begin{equation}\label{epball}
e_{p,k-1}(u)^{\frac {1}{p+k-1}}\leq C(p,k,k-1,n,\mathbb B)e_{p,k}(u)^{\frac {1}{p+k}},
\end{equation}
with
\[
C(p,k,k-1,n,\mathbb B)=p^{\frac {1}{p+k-1}}d(p,k,\mathbb B)^{\frac {1}{p(p+k-1)}}e_{p,k}(|z|^2-1)^{\frac {1}{p+k-1}-\frac {1}{p+k}}.
\]
From~(\ref{epball}) it now follows
\[
C(p,k,l,n,\mathbb B)=C(p,k,k-1,n,\mathbb B)\cdot C(p,k-1,k-2,n,\mathbb B)\dots C(p,l+1,l,n,\mathbb B).
\]
Therefore, for  $p=1$
\[
C(1,k,l,n,\mathbb B)=e_{1,k}(|z|^2-1)^{\frac 1l-\frac 1k}=\left(\frac {(4\pi)^n}{n+1}\right)^{\frac 1l-\frac 1k}.
\]
\end{proof}

\begin{remark}
In~\cite{TrudingerWangPoincare}, Trudinger and Wang used the real Hessian quotient operator $\frac {S_k}{S_l}$ to establish the optimal constant in the Poincar\'e inequality for the real Hessian operator. More precisely, they prove that the optimal constant is attained by the solution $u_0$ of the equation $\frac {S_k(u_0)}{S_l(u_0)}=1$. We suspect that this is also the case in the complex setting. With Theorem~\ref{Piep_optimal} in mind we suspect that the optimal constant is
\[
C(p,k,l,n,\Omega)=\left(e_{p,k}(u_0)\right)^{\frac {1}{p+l}-\frac {1}{p+k}},
\]
where $p>0$ and $u_0\in \mathcal E_{p,k}(\Omega)$ is the unique negative $k$-subharmonic function such that $\operatorname{H}_k(u_0)=\operatorname{H}_l(u_0)$. We refer to~\cite{DinewDoTo, S}, and reference therein for results concerning such functions in Euclidean spaces as well as on compact manifolds.
\end{remark}

\section{A Sobolev type inequality in $m$-hyperconvex domains}\label{sec_sobolev}

Let us first recall the notion of $m$-capacity. Let $n\geq 2$, $1\leq m\leq n$. For an arbitrary bounded domain $\Omega\subset \mathbb C^n$, and for any $K\Subset\Omega$ define
\begin{multline*}
\operatorname{cap}_m(K,\Omega)=\operatorname{cap}_m(K):=\\ \sup\left\{\int_{K}(dd^cu)^m\wedge(dd^c|z|^2)^{n-m}: u\in \mathcal {SH}_m(\Omega), -1\leq u\leq 0 \right\}.
\end{multline*}

The following lemma was proved by Dinew and Ko\l odziej~\cite{DinewKolodziej}.

\begin{lemma}\label{dk}
Let $n\geq 2$, $1\leq m\leq n$, and let $\Omega\subset \mathbb C^n$ be a $m$-hyperconvex domain. Then for $1<\alpha<\frac {n}{n-m}$ there exists a constant $C(\alpha)>0$ such that for any $K\Subset\Omega$,
\[
V_{2n}(K)\leq C(\alpha)\operatorname{cap}_m^{\alpha}(K).
\]
\end{lemma}

We will also need the following two lemmas.

\begin{lemma}\label{sublevel}
Let $n\geq 2$, $1\leq m\leq n$, $p\geq 0$, and let $\Omega\subset \mathbb C^n$ be a $m$-hyperconvex domain. For $u\in \mathcal E_{p,m}(\Omega)$, and any $s>0$, it holds
\[
\operatorname{cap}_m(\{u<-s\})\leq 2^{m+p}s^{-m-p}e_{p,m}(u).
\]
\end{lemma}
\begin{proof}
By~\cite{ThienPhD,Thien18} we have for any $s,t>0$
\begin{multline*}
t^m\operatorname{cap}_m(\{u<-s-t\})\leq \int_{\{u<-s\}}(dd^cu)^m\wedge(dd^c|z|^2)^{n-m}\\
\leq s^{-p}\int_{\{u<-s\}}(-u)^p(dd^cu)^m\wedge(dd^c|z|^2)^{n-m}\leq s^{-p}e_{p,m}(u).
\end{multline*}
Taking $t=s$ we get
\[
\operatorname{cap}_m(\{u<-2s\})\leq s^{-p-m}e_{p,m}(u).
\]
\end{proof}

\begin{lemma}\label{lp}
Let $n\geq 2$, $1\leq m\leq n$, $p\geq 0$, and assume that $\Omega\subset \mathbb C^n$ is a $m$-hyperconvex domain.  Then we have
 that $\mathcal E_{p,m}(\Omega)\subset L^q(\Omega)$, for any $0<q<\frac {n(m+p)}{n-m}$.
\end{lemma}
\begin{proof}
Assume first that $u\in \mathcal E^0_{m}(\Omega)$, and let $p\geq 0$. Let us define
\[
\lambda(s)=V_{2n}(\{u<-s\}).
\]
Then by Lemma~\ref{dk}, and Lemma~\ref{sublevel}, we get that for $0<\alpha<\frac {n}{n-m}$
\[
\lambda(s)\leq C_1\operatorname{cap}_m^{\alpha}(\{u<-s\})\leq C_2s^{-(m+p)\alpha}e_{p,m}(u)^{\alpha},
\]
where $C_1$ and $C_2$ are constants not depending on $u$. For $q>0$ we then have
\begin{multline}\label{lp_1}
\int_{\Omega}(-u)^q\,dV_{2n}=q\int_{0}^{\infty}s^{q-1}\lambda(s)\,ds=q\int_{0}^{1}s^{q-1}\lambda(s)\,ds+q\int_{1}^{\infty}s^{q-1}\lambda(s)\,ds\\
\leq qV_{2n}(\Omega)+C_3e_{p,m}(u)^{\alpha}\int_1^{\infty}s^{q-1-(m+p)\alpha}\,ds< \infty \, \Leftrightarrow \, q<(m+p)\alpha< \frac {n(m+p)}{n-m},
\end{multline}
where $C_3$ is a constant not depending on $u$. From~(\ref{lp_1}) we have $\int_{\Omega}(-u)^q\,dV_{2n}<\infty$ if, and only if,
\[
q<(m+p)\alpha< \frac {n(m+p)}{n-m}.
\]
Next, if we take a function $u\in \mathcal E_{p,m}(\Omega)$, then there exists a decreasing sequence $u_j\in \mathcal E^0_{m}(\Omega)$ such that $u_j\searrow u$ and $\sup_j e_{p,m}(u)<\infty$. By the first part of the proof there are constants $A,B$ do not depending on $u_j$ such that
\[
\|u_j\|_{L^q}\leq A+Be_{p,m}(u_j)^{\alpha},
\]
and by passing to the limit we get
\[
\|u\|_{L^q}\leq A+B\sup_je_{p,m}(u_j)^{\alpha}<\infty.
\]
\end{proof}

Now we can state and prove the Sobolev type inequality in arbitrary $m$-hyperconvex domains.

\begin{theorem}\label{sob} Let $n\geq 2$, $1\leq m\leq n$, and $p\geq0$. Assume that $\Omega$ be a bounded $m$-hyperconvex domain in $\mathbb C^n$. There exists a constant $C(p,q,m,n,\Omega)>0$, depending only on $p$, $q$, $m$, $n$, and $\Omega$ such that for any function $u\in \mathcal E_{p,m}(\Omega)$, and for $0<q<\frac {(m+p)n}{n-m}$, we have
\begin{equation}\label{lq theorem3}
\|u\|_{L^q}\leq C(p,q,m,n,\Omega) e_{p,m}(u)^{\frac {1}{m+p}}\, .
\end{equation}
\end{theorem}
\begin{proof}
This follows from Lemma~\ref{lp} and Theorem~\ref{qB}.
\end{proof}

We now give examples that shows that the following inequalities are not in general possible:
\[
\begin{aligned}
e_{p,m}(u)^{\frac {1}{m+p}}& \leq C\|u\|_{L^q} \ (\text{Example}~\ref{ex1})\\
\|u\|_{L^{\infty}} &\leq C e_{p,m}(u)^{\frac {1}{m+p}} \ (\text{Example}~\ref{ex2}) \\
e_{p,m}(u)^{\frac {1}{n+p}}& \leq C \|u\|_{L^{\infty}} \ (\text{Example}~\ref{ex3})\, .
\end{aligned}
\]

\begin{example}\label{ex1}
Consider the following functions defined on the unit ball $\mathbb{B}$ in $\C^n$
\[
u_j(z)=\frac 1{j^\alpha}\max\left(1-|z|^{2-\frac {2n}{m}},1-j^{\beta}\right)\, .
\]
Then we have
\[
u_j(z)=
\begin{cases}
\frac{1}{j^{\alpha}}\left(1-|z|^{2-\frac {2n}{m}}\right) & \text{ if }   j^{\beta\frac {m}{2m-2n}}\leq |z|\leq 1 \\[2mm]
\frac {1-j^{\beta}}{j^{\alpha}} & \text{ if } 0\leq |z|\leq j^{\beta\frac {m}{2m-2n}}\, .
\end{cases}
\]
Hence, if $\beta>\alpha \frac {p+m}{p}$, then
\[
e_{p,m}(u)=c(n,m)\frac {1}{j^{\alpha(m+p)}}(j^{\beta} -1)^p\to\infty,\qquad \text{as } j\to \infty,
\]
and
\[
c(n,m)=\frac {2\pi^n(\frac{n}{m}-1)^m}{m!(n-m)!}
\]
(see~\cite{WanWang} for details).

On the other hand, one can check that if $0<q<\frac {mn}{(n-m)(\beta-\alpha)}$, then
\[
\|u_j\|_{L^q}^q\simeq j^{\beta q-\alpha q +\frac {mn}{m-n}}\to 0,
\]
as $j\to \infty$. This shows that we can not in general have
\[
e_{p,m}(u)^{\frac {1}{m+p}} \leq C\|u\|_{L^q}.
\]
\hfill{$\Box$}
\end{example}

\begin{example}\label{ex2}
Similarly as in Example~\ref{ex1} consider the following functions defined on the unit ball $\mathbb{B}$ in $\C^n$
\[
u_j(z)=\frac 1{j^{\frac {p}{m+p}}}\max\left(1-|z|^{2-\frac {2n}{m}},-j\right)\, .
\]
Then we have that
\[
\|u_j\|_{L^{\infty}}=-u_j(0)=j^{\frac {m}{m+p}}\to \infty\qquad\text{as } j\to \infty,
\]
and at the same time
\[
e_{p,m}(u_j)=c(n,m) j^p\left(\frac 1{j^{\frac {p}{m+p}}}\right)^{m+p}=c(n,m).
\]
Hence, a contradiction is obtained. Thus, we can not in general have
\[
\|u\|_{L^{\infty}} \leq C e_{p,m}(u)^{\frac {1}{m+p}}.
\]

\hfill{$\Box$}
\end{example}

\begin{example}\label{ex3}
Similarly as before we consider the following functions defined on the unit ball $\mathbb{B}$ in $\C^n$
\[
u_j(z)=j\max\left(1-|z|^{2-\frac {2n}{m}},-\frac 1j\right)\, .
\]
Then we have that $\|u_j\|_{L^{\infty}}=-u_j(0)=1$,  but at the same time
\[
e_{p,m}(u_j)=c(n,m) j^{m+p}\left(\frac 1j\right)^p=c(n,m)j^m\to \infty\qquad \text{as } j\to \infty.
\]
This shows that we can not in general have
\[
e_{p,m}(u)^{\frac {1}{n+p}} \leq C \|u\|_{L^{\infty}}.
\]
\hfill{$\Box$}
\end{example}

 As an immediate consequence of Theorem~\ref{sob} is that we get that B\l ocki's integrability conjecture is true for functions in the Cegrell class $\mathcal E_m(\Omega)$ (Corollary~\ref{cor}). Before stating this result let us recalling the definition of $\mathcal E_m(\Omega)$. Let $\Omega$ be a bounded $m$-hyperconvex domain in $\mathbb C^n$. We say that $u\in \mathcal E_m(\Omega)$ if for any $\omega\Subset\Omega$ there exists $u_{\omega}\in \mathcal E_{0,m}(\Omega)$ such that $u=u_{\omega}$ on $\omega$.

\begin{corollary}\label{cor}
Let $n\geq 2$, $1\leq m\leq n$, and let $\Omega$ be a bounded $m$-hyperconvex domain in $\mathbb C^n$. Then $\mathcal E_m(\Omega)\subset L^q_{loc}(\Omega)$, for $0<q<\frac {nm}{n-m}$.
\end{corollary}

\end{document}